\documentclass[preprint,12pt]{elsarticle}

\usepackage[T1]{fontenc}
\usepackage{lmodern}
\usepackage{microtype}
\usepackage{amsmath,amssymb,amsfonts,amsthm,mathtools}
\usepackage{bm}
\usepackage{enumitem}
\usepackage{booktabs}
\usepackage[margin=1.15in]{geometry}
\usepackage{xcolor}
\usepackage{hyperref}
\usepackage[nameinlink,noabbrev]{cleveref}
\hypersetup{
	pdfauthor={Xuelong Gu and Qi Wang},
	pdftitle={A generalized skew-gradient embedding framework for thermodynamically consistent systems},
	colorlinks=false,
	linkbordercolor={1 0 0},
	citebordercolor={0 1 0},
	urlbordercolor={0 0 1},
	pdfborder={0 0 1}
}

\makeatletter
\def\ps@pprintTitle{%
	\let\@oddhead\@empty
	\let\@evenhead\@empty
	\let\@oddfoot\@empty
	\let\@evenfoot\@oddfoot}
\makeatother

\newtheorem{theorem}{Theorem}[section]
\newtheorem{proposition}[theorem]{Proposition}

\newtheorem{definition}[theorem]{Definition}
\newtheorem{remark}[theorem]{Remark}

\newcommand{\R}{\mathbb R}

\newcommand{\calC}{\mathcal C}

\newcommand{\Om}{\Omega}
\newcommand{\dd}{\mathrm d}
\newcommand{\ip}[2]{\left\langle #1,#2\right\rangle}
\newcommand{\norm}[1]{\left\|#1\right\|}
\newcommand{\abs}[1]{\left|#1\right|}
\newcommand{\Span}{\operatorname{span}}
\newcommand{\diver}{\operatorname{div}}
\newcommand{\grad}{\nabla}
\newcommand{\bfu}{\bm u}
\newcommand{\bfv}{\bm v}
\newcommand{\bfw}{\bm w}

\begin{document}

\begin{frontmatter}

	\title{Generalized skew-gradient embedding for thermodynamically consistent systems}

	\author[usc]{Xuelong Gu\fnref{funding}}
	\author[usc]{Qi Wang\corref{cor1}\fnref{funding}}
	\ead{QWANG@math.sc.edu}
	\cortext[cor1]{Corresponding author.}
	\fntext[funding]{X.G. is supported by NSF award OIA-2242812. Q.W. is partially supported by NSF awards DMS-2038080 and OIA-2242812, and DOE award DE-SC0025229.}
	\address[usc]{Department of Mathematics, University of South Carolina, Columbia, SC 29208, USA}

	\begin{abstract}
		GENERIC describes thermodynamic evolution through coupled reversible and irreversible operators subject to energy--entropy degeneracy conditions. The skew-gradient embedding (SGE) framework~\cite{GuWangSGE2025} embeds a zero-energy contribution term in a rank-two skew-symmetric matrix, rewriting the system as a generalized gradient flow. This matrix may be evaluated at previous time levels without losing skew-symmetry or the discrete energy law, and the resulting explicit treatment often decouples multiphysics variables. The rank-two representation is not unique because the dynamics fixes the action of the matrix only along the thermodynamic force. We characterize the resulting affine family of admissible two-forms and call them generalized skew-gradient embeddings (GSGE). Weighted least squares selects a unique minimum-Hilbert--Schmidt gauge and recovers SGE in the native metric. Regularization keeps the gauge well-defined when the thermodynamic force is small or vanishes, while a projection-based construction yields gauges that preserve multiple prescribed invariants simultaneously. For full GENERIC systems, an invariant-preserving selection within this family retains the entropy-production law and total-energy conservation simultaneously. For isothermal systems, we give a rank-two Jacobi criterion for the gauge to define a Poisson structure, which permits GENERIC integrators in the operator sense. We illustrate the framework with two fluid systems. For the incompressible Navier--Stokes equations, a compatible MAC spatial discretization satisfies this discrete Jacobi criterion; combined with the implicit midpoint rule in time, it yields a fully discrete rank-two GENERIC representation with an exact discrete free-energy law. For the Cahn--Hilliard--Navier--Stokes system, a regularized GSGE--BDF2 scheme preserves mass, dissipates the discrete free energy unconditionally, and permits a decoupled implementation.
	\end{abstract}


\end{frontmatter}

\section{Introduction}
\label{sec:introduction}

Thermodynamically consistent models in classical electrodynamics, fluid and solid mechanics, quantum mechanics, complex fluids, phase-field hydrodynamics, and statistical physics often arise from conservation laws coupled with constitutive relations \cite{Morrison-1986-Joined,port,Wang2021GOP,Hong-2024-JCP,Jiang-2024-yield}. The GENERIC formalism separates their reversible and irreversible parts as follows \cite{GrmelaOttinger1997a,OttingerGrmela1997b,OttingerBook2005,Ottinger-2018-GENERIC-Integrators,Bloch-2024-Metriplectic}:
\begin{equation}
	\label{eq:intro-onsager-zec}
	\partial_t\Phi=L(\Phi)\,\grad E(\Phi)+M(\Phi)\,\grad S(\Phi).
\end{equation}
Here $\Phi$ is the collection of thermodynamic variables, $E$ is the total energy, and $S$ is the entropy. At each state, $L(\Phi),M(\Phi):H\to H$. The operator $L$ is Poisson, while $M$ is symmetric positive semidefinite. They satisfy the GENERIC degeneracy conditions
\begin{equation}
	\label{eq:generic-degeneracy}
	L(\Phi)\,\grad S(\Phi)=0,
	\quad
	M(\Phi)\,\grad E(\Phi)=0.
\end{equation}
The first condition states that reversible motion preserves entropy, and the second states that irreversible motion preserves total energy. Skew-symmetry and \eqref{eq:generic-degeneracy} give
\begin{equation}
	\label{eq:intro-energy-law}
	\begin{aligned}
		\tfrac{\dd}{\dd t}E(\Phi(t))
		 & =\ip{\grad E}{L\,\grad E}+\ip{\grad E}{M\,\grad S}=0,\\
		\tfrac{\dd}{\dd t}S(\Phi(t))
		 & =\ip{\grad S}{L\,\grad E}+\ip{\grad S}{M\,\grad S}
		=\ip{\grad S}{M\,\grad S}\ge0.
	\end{aligned}
\end{equation}
At a fixed temperature $\theta>0$, set $F=E-\theta S$ \cite{Onsager1931a,Onsager1931b,Wang2021GOP}. The degeneracy conditions imply $L\,\grad F=L\,\grad E$ and $M\,\grad S=-\tfrac{1}{\theta}M\,\grad F$. Hence \eqref{eq:intro-onsager-zec} reduces to
\begin{equation}
	\label{eq:isothermal-reduction}
	\partial_t\Phi=-\tfrac{1}{\theta}M(\Phi)\,\grad F(\Phi)+L(\Phi)\,\grad F(\Phi),
\end{equation}
and $\tfrac{\dd}{\dd t}F=-\tfrac{1}{\theta}\ip{\grad F}{M\,\grad F}\le0$.

At the continuous level, the reversible term is $L\,\grad E$ and satisfies $\ip{\grad S}{L\,\grad E}=0$. SGE \cite{GuWangSGE2025} exploits this orthogonality by embedding the reversible field in the rank-two skew-symmetric matrix
\begin{equation}
	\label{eq:intro-sge-choice}
	J(\Phi)
	=\tfrac{\grad S(\Phi)\wedge L(\Phi)\grad E(\Phi)}
	{\norm{\grad S(\Phi)}^2},
	\quad J(\Phi)\grad S(\Phi)=L(\Phi)\grad E(\Phi).
\end{equation}
Pulling the reversible field back to the entropy force therefore rewrites full GENERIC as the generalized gradient flow
\[
	\partial_t\Phi
	=\bigl(M(\Phi)+J(\Phi)\bigr)\grad S(\Phi).
\]
Its symmetric and skew parts are driven by the same force $\grad S$. Treating the two defining profiles explicitly preserves skew cancellation and hence the entropy law. The resulting rank-two update requires a fixed number of solves with the same dissipative operator and a small scalar system; in the Navier--Stokes and CHNS schemes of \cite{GuWangSGE2025}, only two scalar coefficients couple the otherwise decoupled subproblems.

The matrix $J(\Phi)$ in \eqref{eq:intro-sge-choice} is not unique, since the dynamics fixes only its action on $\grad S(\Phi)$. This observation motivates the present work: we seek all two-forms $\omega$ such that
\begin{equation}
	\label{eq:intro-gsge}
	\bigl(\iota_{\grad S(\Phi)}\omega(\Phi)\bigr)^\sharp
	=L(\Phi)\grad E(\Phi).
\end{equation}
We call these representations generalized skew-gradient embeddings (GSGE). The admissible two-forms form an affine space because the dynamics fixes only their contraction with the thermodynamic force. The SGE matrix \eqref{eq:intro-sge-choice} corresponds to one representative of this gauge freedom. We introduce three criteria for selecting a representative: weighted least-squares optimality, preservation of energy or other invariants, and, for isothermal systems, compatibility with Poisson geometry.

The first principle selects gauges by least-squares optimality. For any symmetric positive definite gauge map $A$, the operator-weighted two-form
\begin{equation}
	\label{eq:intro-A-gauge}
	\omega_A
	=\tfrac{(A\grad S(\Phi))^\flat\wedge (L(\Phi)\grad E(\Phi))^\flat}
	{\ip{A\grad S(\Phi)}{\grad S(\Phi)}}
\end{equation}
is the unique minimum-Hilbert--Schmidt representative in the $A$-metric. The native metric recovers SGE. The same principle gives regularized approximations and projections of non-neutral residuals, and it justifies the regularized differential gauge used later for CHNS.

The second principle builds prescribed invariants $\calC_1,\ldots,\calC_m$ into the gauge. Projecting $\grad S$ onto the $A$-orthogonal complement of their gradients gives $\widetilde X$ and
\[
	\omega_{A,\calC}
	=\tfrac{(A\widetilde X)^\flat\wedge(L\,\grad E)^\flat}
	{\ip{A\widetilde X}{\grad S}},
	\quad
	\bigl(\iota_{\grad S}\omega_{A,\calC}\bigr)^\sharp=L\,\grad E,
	\quad
	\iota_{\grad\calC_\alpha}\omega_{A,\calC}=0.
\]
Thus the selected gauge preserves all prescribed invariants; taking $\calC_1=E$ gives $\iota_{\grad E}\omega_{A,\calC}=0$.

For the isothermal reduction, we also ask when the selected gauge is Poisson \cite{CrainicFernandesMarcut2021}. A rank-two candidate is $L_2=Y\wedge(L\,\grad F)$, acting by
\[
	L_2v=\ip{Y}{v}L\,\grad F-\ip{L\,\grad F}{v}Y,
	\quad v\in H .
\]
The identities $L_2\grad F=L\,\grad F$ and $[Y,L\,\grad F]\in\Span\{Y,L\,\grad F\}$ give the reversible representation and, in finite dimensions, the rank-two Poisson criterion. The Navier--Stokes example combines a compatible marker-and-cell (MAC) spatial discretization \cite{Harlow-1965-MAC,Gong-2018-StaggeredCHNS} with the implicit midpoint rule to obtain a fully discrete isothermal GENERIC representation.

The main contributions are summarized as follows:
\begin{enumerate}[label=(\roman*)]
	\item We formulate GSGE for the GENERIC reversible action $L\,\grad E$ and its isothermal free-energy reduction, and characterize the resulting affine gauge space.

	\item We establish a unified least-squares principle for operator-weighted gauges, including regularized gauges and residual-projection corrections. These constructions retain the entropy-production law but do not, in general, preserve total energy. For full GENERIC, a projection-based construction additionally preserves total energy and any further prescribed invariants.

	\item For the isothermal reduction, we give a finite-dimensional necessary and sufficient condition for a rank-two gauge to satisfy the Jacobi identity and prove local existence with prescribed invariants.

	\item For the incompressible Navier--Stokes equations, a compatible MAC spatial discretization satisfies the discrete rank-two Jacobi identity. The implicit midpoint rule in time then yields a fully discrete scheme with a rank-two GENERIC representation and an exact discrete free-energy law.

	\item For the isothermal CHNS system, we propose a GSGE--BDF2 scheme based on a three-parameter regularized differential gauge. The scheme preserves mass, dissipates the discrete free energy unconditionally, and admits a decoupled implementation. Two parameter limits recover the SGE gauge and a gradient-weighted mass-preserving gauge, respectively.
\end{enumerate}

The remainder of the paper is organized as follows. \Cref{sec:gsge} introduces the finite-dimensional setting, reviews SGE, and formulates GSGE. \Cref{sec:selection} develops gauge-selection principles based on least-squares optimality, invariant preservation, and Poisson geometry. \Cref{sec:examples} applies the theory to the Navier--Stokes and CHNS systems. \Cref{sec:conclusion} contains concluding remarks.

\section{Preliminaries}
\label{sec:gsge}

\subsection{Notation}
\label{subsec:notation}

Let $(H,\ip{\cdot}{\cdot})$ be a finite-dimensional real inner-product space with dual space $H'$. The structural theory is finite-dimensional. For PDE models, $H$ denotes the spatially discrete state space, and the continuum formulas in \Cref{sec:examples} only indicate the identities to be preserved by compatible discretizations. Denote the Riesz isomorphism and its inverse by the musical maps
\begin{equation*}
	\flat:H\to H',\quad
	v\mapsto v^\flat=\ip{v}{\cdot},
	\quad
	\sharp=\flat^{-1}:H'\to H.
\end{equation*}
For a differentiable functional $Q$, write $\grad Q:=(\dd Q)^\sharp$.
We use $\ip{\cdot}{\cdot}$ for the inner product on $H$ and the duality pairing between $H'$ and $H$. The GENERIC operators satisfy $L,M:H\to H$.
Let $\Lambda^2H'$ denote the space of skew-symmetric bilinear forms $\omega:H\times H\to\R$. For $\alpha,\beta\in H'$, define their wedge by
\begin{equation}
	\label{eq:wedge-def}
	(\alpha\wedge\beta)(u,v)=\alpha(u)\beta(v)-\alpha(v)\beta(u),
\end{equation}
Define the contraction by $\iota_v\omega=\omega(v,\cdot)\in H'$. In particular,
$\iota_u(\alpha\wedge\beta)=\alpha(u)\beta-\beta(u)\alpha$.

For $a,b\in H$, we also write $a\wedge b$ for the skew-symmetric map defined by
\begin{equation}
	\label{eq:vector-wedge}
	(a\wedge b)\,u=\ip{a}{u}\,b-\ip{b}{u}\,a .
\end{equation}
All geometric statements are understood on an open subset of $H$. Affine constraints are handled by identifying the affine state space with its associated linear space, endowed with the induced inner product. The Lie bracket of vector fields is
\begin{equation*}
	[X,Z]=DZ[X]-DX[Z].
\end{equation*}
For a smooth functional $\mathcal Q$, write $X(\mathcal Q):=\dd\mathcal Q(X)$; then $[X,Z](\mathcal Q)=X(Z(\mathcal Q))-Z(X(\mathcal Q))$.
A bivector field is represented through the inner product by a skew-symmetric map $L:H\to H$. In particular, a decomposable bivector $Y\wedge Z$ acts by
\begin{equation}
	\label{eq:bivector-action}
	(Y\wedge Z)v=\ip{Y}{v}Z-\ip{Z}{v}Y,
	\quad v\in H.
\end{equation}
It induces the bracket $\{\mathcal A,\mathcal B\}_L=\ip{\grad\mathcal A}{L\,\grad\mathcal B}$
and defines a Poisson structure if for all  $\mathcal A,\mathcal B,\mathcal C$,
\begin{equation*}
	\{\mathcal A,\{\mathcal B,\mathcal C\}_L\}_L+\{\mathcal B,\{\mathcal C,\mathcal A\}_L\}_L+\{\mathcal C,\{\mathcal A,\mathcal B\}_L\}_L=0.
\end{equation*}
Equivalently, in local coordinates, for all $i,j,k$,
\begin{equation}
	\label{eq:jacobi-coordinates}
	\sum\limits_{\ell}
	L^{i\ell}\partial_\ell L^{jk}
	+L^{j\ell}\partial_\ell L^{ki}
	+L^{k\ell}\partial_\ell L^{ij} = 0.
\end{equation}
\subsection{Review of the skew-gradient embedding}
\label{subsec:sge-review}
We first recall the algebraic SGE construction used in both formulations.
\begin{theorem}[{\cite{GuWangSGE2025}}]
	\label{thm:sge-review}
	Let $X(\Phi),Z(\Phi)\in H$ satisfy $X(\Phi)\ne0$ and $\ip{X(\Phi)}{Z(\Phi)}=0$. Then
	\begin{equation*}
		\omega_*(\Phi)=\tfrac{X(\Phi)^\flat\wedge Z(\Phi)^\flat}{\norm{X(\Phi)}^2},
	\end{equation*}
	where \(\omega_*(\Phi)\) satisfies
	\[
		\bigl(\iota_{X(\Phi)}\omega_*(\Phi)\bigr)^\sharp=Z(\Phi).
	\]
\end{theorem}
For full GENERIC, take $X=\grad S$ and $Z=L\,\grad E$; for the isothermal reduction, take $X=\grad F$ and $Z=L\,\grad F$. The skew form separates the reversible profile from the irreversible discretization. In the isothermal case, the latter may use discrete-gradient methods, stabilization, convex splitting, EQ/SAV-type approaches, averaged-vector-field methods, or supplementary-variable formulations \cite{Gonzalez1996,FurihataMatsuo2010,CS,CS1,CS2,Feng-2013,Hou-2020,EQ1,EQ3,EQ4,SAV-Shen,Jiang-2022,Zhang-2022-GSAV,AVF1,Mcl-1999-DG,Gong-2021}. Freezing the profiles defining $\omega_*$ preserves skew cancellation and often decouples multiphysics variables.
Furthermore, the operator associated with $\omega_*$ need not be assembled explicitly. Its action is
\begin{equation}
	\label{eq:matrix-free-S}
	\bigl(\iota_v\omega_*\bigr)^\sharp
	=\tfrac{\ip{X(\Phi)}{v}\,Z(\Phi)
	-\ip{Z(\Phi)}{v}\,X(\Phi)}
	{\norm{X(\Phi)}^2},
\end{equation}
which requires only two inner products and two vector updates. All gauges constructed below retain this matrix-free structure.

\subsection{Generalized skew gradient embedding}
\label{subsec:gsge-gauge}

The SGE construction is a special case of the following definition.

\begin{definition}
	\label{def:gsge}
	A two-form $\omega\in\Lambda^2H'$ is an \emph{admissible ZEC gauge} at the state $\Phi$ if
	\begin{equation}
		\label{eq:compatibility}
		\bigl(\iota_{\grad S(\Phi)}\omega(\Phi)\bigr)^\sharp
		=L(\Phi)\grad E(\Phi).
	\end{equation}
	The corresponding full GENERIC system is the entropy-driven generalized gradient flow
	\begin{equation}
		\label{eq:gsge-system}
		\partial_t\Phi=\bigl(\iota_{\grad S(\Phi)}\omega(\Phi)\bigr)^\sharp+M(\Phi)\,\grad S(\Phi).
	\end{equation}
\end{definition}
For full GENERIC, skewness gives the entropy law directly:
\begin{equation*}
	\tfrac{\dd}{\dd t}S
	=\omega(\grad S,\grad S)+\ip{\grad S}{M\,\grad S}
	=\ip{\grad S}{M\,\grad S}\ge0.
\end{equation*}
At states with $\grad S\ne0$, \Cref{thm:sge-review} supplies an admissible gauge because $\ip{\grad S}{L\,\grad E}=0$.
\begin{proposition}
	\label{prop:affine-gauge}
	Suppose $X\in H$ is nonzero, $r\in H'$ satisfies $r(X)=0$, and $\omega_0(X,v)=r(v)$ for every $v\in H$. Then every gauge satisfying the same condition is of the form
	\begin{equation*}
		\omega=\omega_0+\zeta,
		\quad
		\zeta(X,v)=0,\quad \forall v\in H,
	\end{equation*}
\end{proposition}

\begin{proof}
	The compatibility condition is affine in $\omega$, so two solutions differ by a two-form vanishing on $(X,v)$ for every $v\in H$.
\end{proof}

Thus the thermodynamic balance laws leave an affine gauge freedom.

\begin{remark}[Isothermal reduction]
	For an isothermal system, replace $\grad S$ and $L\,\grad E$ in \Cref{def:gsge} by $\grad F$ and $L\,\grad F$, respectively. The resulting free-energy-driven system is
	\begin{equation*}
		\partial_t\Phi=-\tfrac{1}{\theta}M(\Phi)\grad F(\Phi)
		+\bigl(\iota_{\grad F(\Phi)}\omega(\Phi)\bigr)^\sharp,
	\end{equation*}
	and skewness gives
	\begin{equation*}
		\tfrac{\dd}{\dd t}F
		=-\tfrac{1}{\theta}\ip{\grad F}{M\,\grad F}+\omega(\grad F,\grad F)
		=-\tfrac{1}{\theta}\ip{\grad F}{M\,\grad F}\le0.
	\end{equation*}
	At states with $\grad F\ne0$, admissible gauges exist because $\ip{\grad F}{L\,\grad F}=0$, and \Cref{prop:affine-gauge} gives the corresponding affine family.
\end{remark}

\section{Gauge selection}
\label{sec:selection}

Selecting a representative from this family requires an additional criterion. We consider three: least-squares minimality, preservation of prescribed invariants, and compatibility with Poisson geometry.

\subsection{A unified least-squares principle}
\label{subsec:unified-ls}

Let $A:H\to H$ be self-adjoint and positive definite, with norms
\begin{equation*}
	\norm{v}_A^2=\ip{Av}{v}\quad\text{on }H,
	\quad
	\norm{r}_{A^{-1}}^2=\ip{r^\sharp}{A^{-1}r^\sharp}\quad\text{on }H'.
\end{equation*}
For $\omega,\xi\in\Lambda^2H'$, define the $A$-weighted Frobenius inner product and norm by
\begin{equation}
	\label{eq:two-form-Frobenius}
	\ip{\omega}{\xi}_{\mathrm F,A}
	:=\sum_{i,j=1}^{\dim H}\omega(e_i,e_j)\xi(e_i,e_j),
	\quad
	\norm{\omega}_{\mathrm F,A}^2:=\ip{\omega}{\omega}_{\mathrm F,A},
\end{equation}
where $\{e_i\}$ is any $A$-orthonormal basis. This is the Frobenius structure induced by the $A$-inner product on covariant two-tensors, restricted to $\Lambda^2H'$. The inner-product properties, and hence the norm properties, are immediate once basis independence is established. We verify the latter briefly. Let $\{f_k\}$ be another $A$-orthonormal basis and write $f_k=\sum_i c_{ik}e_i$. Since $\sum_k c_{ik}c_{pk}=\delta_{ip}$, bilinearity gives
	\begin{align*}
		\sum_{k,l}\omega(f_k,f_l)\xi(f_k,f_l)
		&=\sum_{i,j,p,q}\omega(e_i,e_j)\xi(e_p,e_q)\,
		\delta_{ip}\delta_{jq} \\
		&=\sum_{i,j}\omega(e_i,e_j)\xi(e_i,e_j).
	\end{align*}
Thus \eqref{eq:two-form-Frobenius} is well-defined.

\begin{theorem}[Unified least-squares gauge]
	\label{thm:unified-ls}
	Let $X\ne0$ and let $r\in H'$ satisfy $r(X)=0$. Among all two-forms $\omega\in\Lambda^2H'$ with $\iota_X\omega=r$, the gauge
	\begin{equation}
		\label{eq:A-minimal-form}
		\omega_{A,r}=\tfrac{(AX)^\flat\wedge r}{\ip{AX}{X}}
	\end{equation}
	uniquely minimizes $\norm{\cdot}_{\mathrm F,A}$, and
	\begin{equation*}
		\norm{\omega_{A,r}}_{\mathrm F,A}^2
		=\tfrac{2\norm{r}_{A^{-1}}^2}{\norm{X}_A^2}.
	\end{equation*}
	For $X=\grad S$ and $r=(L\,\grad E)^\flat$ in full GENERIC, or $X=\grad F$ and $r=(L\,\grad F)^\flat$ in the isothermal reduction, $A=I$ in \eqref{eq:A-minimal-form} gives the SGE gauge. Thus the embedding of \cite{GuWangSGE2025} is the native-metric least-squares gauge.
\end{theorem}

\begin{proof}
	Let $e_1=X/\norm{X}_A$ and extend it to an $A$-orthonormal basis. Since $r(X)=0$, contraction of \eqref{eq:A-minimal-form} with $X$ gives $\iota_X\omega_{A,r}=r$. By \Cref{prop:affine-gauge}, every other admissible gauge has the form $\omega=\omega_{A,r}+\xi$ with $\iota_X\xi=0$, so $\xi(e_1,e_j)=0$ for every $j$. On the other hand, $\omega_{A,r}(e_i,e_j)=0$ whenever $i,j\ge2$. Hence $\ip{\omega_{A,r}}{\xi}_{\mathrm F,A}=0$, and
	\begin{equation*}
		\norm{\omega}_{\mathrm F,A}^2
		=\norm{\omega_{A,r}}_{\mathrm F,A}^2
		+\norm{\xi}_{\mathrm F,A}^2.
	\end{equation*}
	This proves uniqueness and minimality. The only nonzero components of $\omega_{A,r}$ have one index equal to $1$, and $r(e_1)=0$; therefore \eqref{eq:two-form-Frobenius} gives the stated norm.
\end{proof}

\begin{theorem}[Regularized gauge]
	Let $X\in H$ and let $r\in H'$ satisfy $r(X)=0$. For $\sigma>0$, the problem
	\begin{equation*}
		\min_{\omega\in\Lambda^2H'}
		\left\{
		\tfrac12\norm{\iota_X\omega-r}_{A^{-1}}^2
		+\tfrac{\sigma}{4}\norm{\omega}_{\mathrm F,A}^2
		\right\}
	\end{equation*}
	has the unique solution
	\begin{equation}
		\label{eq:regularized-solution}
		\omega^{\sigma}
		=\tfrac{(AX)^\flat\wedge r}{\ip{AX}{X}+\sigma}.
	\end{equation}
	If $X\ne0$, then $\omega^\sigma$ converges to the minimum-norm gauge \eqref{eq:A-minimal-form} as $\sigma\downarrow0$.
\end{theorem}

\begin{proof}
	If $X=0$, the unique minimizer is $\omega^\sigma=0$. Suppose $X\ne0$. Set $q=\iota_X\omega$. Then $q(X)=0$, and \Cref{thm:unified-ls} reduces the problem to
	\begin{equation*}
		\tfrac12\norm{q-r}_{A^{-1}}^2
		+\tfrac{\sigma}{2\norm{X}_A^2}\norm{q}_{A^{-1}}^2
	\end{equation*}
	over $q(X)=0$. Its unique minimizer is $q=\tfrac{\norm{X}_A^2}{\norm{X}_A^2+\sigma}r$. Substitution into \eqref{eq:A-minimal-form} gives \eqref{eq:regularized-solution}.
\end{proof}

\begin{remark}
	\label{rem:regularized-gauge}
	Contraction of \eqref{eq:regularized-solution} gives
	\begin{equation*}
		\iota_X\omega^\sigma
		=\tfrac{\ip{AX}{X}}{\ip{AX}{X}+\sigma}\,r,
		\quad
		\norm{\iota_X\omega^\sigma-r}_{A^{-1}}
		\le \tfrac{\sigma}{c}\norm{r}_{A^{-1}}
	\end{equation*}
	whenever $\ip{AX}{X}\ge c>0$. Thus, for $X=\grad S$ and $r=(L\,\grad E)^\flat$,
	\begin{equation*}
		\bigl(\iota_{\grad S(\Phi)}\omega^\sigma\bigr)^\sharp
		=L(\Phi)\grad E(\Phi)+O(\sigma),
		\quad
		\omega^\sigma(\grad S(\Phi),\grad S(\Phi))=0.
	\end{equation*}
	Choosing $\sigma=O(\rho^p)$ for a method of order $p$ therefore preserves its formal order while retaining exact skew cancellation. The regularized formula is defined at $X=0$, where it reproduces $r$ only if $r=0$.
\end{remark}

The preceding remark regularizes a compatible profile satisfying $r(X)=0$. If a computed profile does not satisfy this condition exactly, the same wedge construction gives its nearest compatible correction.

\begin{proposition}[Optimal ZEC correction]
	\label{prop:best-zec}
	For $X\ne0$ and $r\in H'$, define
	\begin{equation*}
		\Pi_X^A r
		:=\iota_X\!\left(\tfrac{(AX)^\flat\wedge r}{\ip{AX}{X}}\right)
		=r-\tfrac{r(X)}{\ip{AX}{X}}(AX)^\flat.
	\end{equation*}
	Then $\Pi_X^A r$ uniquely minimizes $\norm{\xi-r}_{A^{-1}}^2$ over all $\xi\in H'$ satisfying $\xi(X)=0$.
\end{proposition}

\begin{proof}
	The contraction formula gives $(\Pi_X^A r)(X)=0$. Moreover, $r-\Pi_X^A r$ is a multiple of $(AX)^\flat$, which is $A^{-1}$-orthogonal to every $\zeta\in H'$ satisfying $\zeta(X)=0$, since $\ip{(AX)^\flat}{\zeta}_{A^{-1}}=\zeta(X)$. Thus $\Pi_X^A r$ is the stated orthogonal projection.
\end{proof}

For full GENERIC, the gauges selected above retain the entropy-production law by skewness but need not encode total-energy conservation. We next impose kernel conditions to preserve energy and other invariants.

\subsection{Invariant-preserving gauges}
\label{subsec:invariants}

For full GENERIC, let $\calC_1=E,\calC_2,\ldots,\calC_m$ have linearly independent gradients at the current state and be preserved by $L\,\grad E$. Further functionals may represent mass, momentum, or other constraints. Set
\begin{equation*}
	X_{\calC_\alpha}(\Phi):=(\dd\calC_\alpha(\Phi))^\sharp,
	\quad
	\ip{(L(\Phi)\grad E(\Phi))^\flat}{X_{\calC_\alpha}(\Phi)}=0,
	\quad \alpha=1,\ldots,m.
\end{equation*}
To construct the force leg, form the Gram system
\begin{equation*}
	B_{\alpha\beta}=\ip{A X_{\calC_\beta}(\Phi)}{X_{\calC_\alpha}(\Phi)},
	\quad
	b_\alpha=\ip{A\grad S(\Phi)}{X_{\calC_\alpha}(\Phi)},
\end{equation*}
solve $B\lambda=b$, and set
\begin{equation}
	\label{eq:projected-force-leg}
	\widetilde X(\Phi)=\grad S(\Phi)-\sum_{\alpha=1}^m\lambda_\alpha X_{\calC_\alpha}(\Phi).
\end{equation}

\begin{proposition}
	\label{prop:invariant-gauge}
	Under the assumptions above, if $\widetilde X(\Phi)\ne0$, then the two-form
	\begin{equation}
		\label{eq:projected-invariant-gauge}
		\omega_{A,\calC}
		=\tfrac{(A\widetilde X(\Phi))^\flat\wedge (L(\Phi)\grad E(\Phi))^\flat}
		{\ip{A\widetilde X(\Phi)}{\grad S(\Phi)}}
	\end{equation}
	satisfies
	\begin{equation*}
		\bigl(\iota_{\grad S(\Phi)}\omega_{A,\calC}\bigr)^\sharp
		=L(\Phi)\grad E(\Phi),
		\quad
		\iota_{X_{\calC_\alpha}(\Phi)}\omega_{A,\calC}=0,
		\quad \alpha=1,\ldots,m.
	\end{equation*}
\end{proposition}

\begin{proof}
	The definition gives $\ip{A\widetilde X}{X_{\calC_\alpha}}=0$ and
	\begin{equation*}
		\ip{A\widetilde X}{\grad S}=\ip{A\widetilde X}{\widetilde X}>0.
	\end{equation*}
	The two identities now follow from \eqref{eq:projected-invariant-gauge}, the ZEC condition, and preservation of the $\calC_\alpha$ by $L\,\grad E$.
\end{proof}

The isothermal counterpart follows by replacing $\grad S$ and $L\,\grad E$ with $\grad F$ and $L\,\grad F$, respectively.

\begin{remark}
	Let $R_{\rm irr}=M\,\grad S$ for full GENERIC and $R_{\rm irr}=-\tfrac{1}{\theta}M\,\grad F$ for the isothermal reduction. Suppose in addition that
	\begin{equation*}
		\ip{R_{\rm irr}(\Phi)}{X_{\calC_\alpha}(\Phi)}=0,
		\quad \alpha=1,\ldots,m.
	\end{equation*}
	Then each $\calC_\alpha$ is an invariant of the full system, since
	\begin{equation*}
		\tfrac{\dd}{\dd t}\calC_\alpha(\Phi)
		=\ip{R_{\rm irr}(\Phi)}{X_{\calC_\alpha}(\Phi)}
		+\ip{(L(\Phi)\grad E(\Phi))^\flat}{X_{\calC_\alpha}(\Phi)}=0.
	\end{equation*}
	For $\calC_1=E$, the friction degeneracy and $\iota_{\grad E}\omega=0$ give total-energy conservation. The discrete counterpart requires the corresponding chain rule and orthogonality conditions.
\end{remark}

\subsection{Poisson and GENERIC gauges}
\label{sec:poisson-generic}

For full GENERIC, the GSGE gauge acts on $\grad S$, whereas a Poisson operator generates the reversible field from $\grad E$ and has $\grad S$ in its kernel. These requirements cannot be imposed on the same nonzero operator. We therefore restrict the Poisson reconstruction to the isothermal reduction, where the single force is $\grad F$ and the compatibility condition is $L_2\grad F=L\grad F$.

Set $Z=L\,\grad F$ and consider $L_2=Y\wedge Z$. By \eqref{eq:bivector-action}, $L_2\grad F=Z$ if $\dd F(Y)=1$ and $\dd F(Z)=0$, while the Jacobi identity adds a differential condition.

\begin{proposition}[Rank-two Jacobi criterion]
	\label{prop:rank-two-poisson}
	Let $Y$ and $Z$ be smooth vector fields. Then $L_2=Y\wedge Z$ defines a Poisson structure if and only if
	\begin{equation}
		\label{eq:frobenius-condition}
		Y\wedge Z\wedge[Y,Z]=0,
	\end{equation}
	Equivalently, $[Y,Z]\in\Span\{Y,Z\}$ wherever $Y$ and $Z$ are linearly independent; at points where $Y\wedge Z=0$, \eqref{eq:frobenius-condition} holds automatically. In particular, $[Y,Z]=0$ suffices \cite{CrainicFernandesMarcut2021}.
\end{proposition}

\begin{proposition}[Local isothermal Poisson reconstruction]
	\label{prop:flow-box}
	Let $Z(z_0)\ne0$, and suppose that $\dd F(Z)=0$ and $\dd\calC_a(Z)=0$, $a=1,\ldots,m$, near $z_0$. If $\dd F,\dd\calC_1,\ldots,\dd\calC_m$ are linearly independent at $z_0$, then locally there is a vector field $Y$ with
	\begin{equation*}
		[Y,Z]=0,
		\quad
		\dd F(Y)=1,
		\quad
		\dd\calC_a(Y)=0,\quad a=1,\ldots,m,
	\end{equation*}
	and $L_2=Y\wedge Z$ defines a Poisson structure with
	\begin{equation}
		\label{eq:L-degeneracy}
		L_2\,\grad F=Z,
		\quad
		L_2\,\grad\calC_a=0,\quad a=1,\ldots,m .
	\end{equation}
\end{proposition}

\begin{proof}
	Let $\Psi_t$ be the local flow of $Z$ and $\Sigma$ a section transverse to $Z$ at $z_0$. Independence gives a field $Y_\Sigma$ on $\Sigma$ with $\dd F(Y_\Sigma)=1$ and $\dd\calC_a(Y_\Sigma)=0$. Extend it by $Y_{\Psi_t(z)}=(\Psi_t)_*Y_\Sigma(z)$, so $[Y,Z]=0$. For any $\mathcal Q$ satisfying $Z(\mathcal Q)=0$,
	\begin{equation*}
		Z(Y(\mathcal Q))=Y(Z(\mathcal Q))+[Z,Y](\mathcal Q)=0,
	\end{equation*}
	so the identities on $\Sigma$ propagate along the flow. The conclusion follows from \Cref{prop:rank-two-poisson} and \eqref{eq:bivector-action}.
\end{proof}

The vectors $\grad\calC_a$ lie in the kernel of the constructed Poisson operator. Taking $Z=L\,\grad F$ gives the isothermal application. This rank-two bracket need not coincide with the physical Poisson bracket.

\section{Examples}
\label{sec:examples}

\subsection{Incompressible Navier--Stokes equations and an isothermal GENERIC discretization}
\label{subsec:ns}

Let $\Om$ be a periodic box and $P_\sigma$ the Helmholtz--Leray projection. The incompressible Navier--Stokes equations are \cite{Rorger-2001}
\begin{equation}
	\label{eq:NS}
	\partial_t\bfu+P_\sigma(\bfu\cdot\grad\bfu)=\nu P_\sigma\Delta\bfu.
\end{equation}
We regard the velocity-only viscous system as an isothermal reduction with the thermal variables omitted. On the divergence-free space, set $F(\bfu)=\tfrac12\norm{\bfu}^2$, $-\tfrac{1}{\theta}M\grad F=\nu P_\sigma\Delta\bfu$, and $J_E(\bfu)=-P_\sigma(\bfu\cdot\grad\bfu)$. Since $(J_E(\bfu),\bfu)=0$, the SGE gauge is $S_*=\tfrac{\bfu\wedge J_E}{\norm{\bfu}^2}$. The Euler part also carries the noncanonical Lie--Poisson bracket \cite{Morrison1982,Morrison1998}
\begin{equation}
	\label{eq:euler-lie-poisson}
	\{\mathcal A,\mathcal B\}_{\rm Eul}(\bfu)
	=-\int_\Om \bfu\cdot
	\Bigl[\tfrac{\delta\mathcal A}{\delta\bfu},\tfrac{\delta\mathcal B}{\delta\bfu}\Bigr]\dd x .
\end{equation}
where $[\cdot,\cdot]$ is the Lie bracket of divergence-free vector fields. The following continuum calculation motivates the MAC construction below.

\begin{proposition}[Formal continuum rank-two gauge]
	\label{prop:ns-poisson}
	On $\{\bfu\ne0\}$, the bivector $L_2^{NS}=Y\wedge J_E$, with $Y(\bfu)=\tfrac{\bfu}{\norm{\bfu}^2}$, formally satisfies the Jacobi identity and $L_2^{NS}\,\grad F=J_E$.
\end{proposition}

\begin{proof}
	The identity $L_2^{NS}\grad F=J_E$ follows from $\dd F(Y)=1$, $\dd F(J_E)=0$, and \eqref{eq:bivector-action}. Since $J_E$ is quadratic and free-energy neutral,
	\begin{equation*}
		DJ_E[Y]=\tfrac{2J_E}{\norm{\bfu}^2},
		\quad
		DY[J_E]=\tfrac{J_E}{\norm{\bfu}^2}
		-\tfrac{2(\bfu,J_E)\bfu}{\norm{\bfu}^4}
		=\tfrac{J_E}{\norm{\bfu}^2}.
	\end{equation*}
	Thus $[Y,J_E]=\tfrac{J_E}{\norm{\bfu}^2}$, and \Cref{prop:rank-two-poisson} applies formally.
\end{proof}

This rank-two bracket generates $J_E$ from $F$ but is not the physical bracket \eqref{eq:euler-lie-poisson}. The MAC construction retains the bilinearity and free-energy neutrality used above.

On a uniform two-dimensional MAC grid \cite{Harlow-1965-MAC,Gong-2018-StaggeredCHNS}, store $p$ at cell centers and $(u,v)$ at the corresponding face centers. With the mesh-weighted inner products, define
\begin{equation*}
	\begin{aligned}
		(D_h\bfu)_{i,j}
		 & =\tfrac{u_{i+1/2,j}-u_{i-1/2,j}}{h}
		+\tfrac{v_{i,j+1/2}-v_{i,j-1/2}}{h},   \\
		(G_hp)^x_{i+1/2,j}
		 & =\tfrac{p_{i+1,j}-p_{i,j}}{h},
		\quad
		(G_hp)^y_{i,j+1/2}
		=\tfrac{p_{i,j+1}-p_{i,j}}{h}.
	\end{aligned}
\end{equation*}
Periodic summation by parts gives $G_h=-D_h^{\top}$, equivalently $(D_h\bfv,q)_h=-(\bfv,G_hq)_h$ \cite{ArnoldFalkWinther2006,LipnikovManziniShashkov2014}. Set $V_h=\ker D_h$, let $P_h$ be the orthogonal projection onto $V_h$, and define the componentwise Laplacian by
\begin{equation*}
	\begin{aligned}
		(\Delta_h u)_{i+1/2,j}
		 & =\tfrac{u_{i+3/2,j}-2u_{i+1/2,j}+u_{i-1/2,j}}{h^2}
		+\tfrac{u_{i+1/2,j+1}-2u_{i+1/2,j}+u_{i+1/2,j-1}}{h^2},   \\
		(\Delta_h v)_{i,j+1/2}
		 & =\tfrac{v_{i+1,j+1/2}-2v_{i,j+1/2}+v_{i-1,j+1/2}}{h^2}
		+\tfrac{v_{i,j+3/2}-2v_{i,j+1/2}+v_{i,j-1/2}}{h^2}.
	\end{aligned}
\end{equation*}

For face fields $\bfw$ and $\bm z$, let $N_h(\bfw)\bm z$ and $K_h(\bfw)\bm z$ be the centered advective and conservative MAC approximations, with arithmetic averages placing products on the required faces \cite{Gong-2018-StaggeredCHNS}. Define
\begin{equation}
	\label{eq:Ch-def}
	C_h(\bfw)\bm z
	=\tfrac12\bigl(N_h(\bfw)\bm z+K_h(\bfw)\bm z\bigr)
	\approx\tfrac12\bigl((\bfw\cdot\grad)\bm z+\diver(\bm z\otimes\bfw)\bigr).
\end{equation}
Periodic summation by parts gives $K_h(\bfw)=-N_h(\bfw)^{\top}$, hence $C_h(\bfw)=\tfrac12\bigl(N_h(\bfw)-N_h(\bfw)^{\top}\bigr)$ and $C_h(\bfw)^{\top}=-C_h(\bfw)$. If $D_h\bfw=0$, \eqref{eq:Ch-def} is second-order consistent with $(\bfw\cdot\grad)\bm z$. The semi-discrete scheme is
\begin{equation}
	\label{eq:NS-semidiscrete}
	\dot\bfu=\nu P_h\Delta_h\bfu+J_h(\bfu)
	\quad\text{on }V_h,
\end{equation}
where $J_h(\bfu)=-P_hC_h(\bfu)\bfu$ and $F_h(\bfu)=\tfrac12\norm{\bfu}_h^2$. The operator $\Delta_h$ is symmetric negative semidefinite, and $\norm{\grad_h\bfv}_h^2:=-(\Delta_h\bfv,\bfv)_h$. Skewness and linearity of $C_h$ give
\begin{equation}
	\label{eq:Jh-properties}
	(J_h(\bfu),\bfu)_h=-(C_h(\bfu)\bfu,\bfu)_h=0,
	\quad
	J_h(\lambda\bfu)=\lambda^2J_h(\bfu).
\end{equation}
\begin{theorem}[Semi-discrete rank-two structure]
	\label{thm:ns-semidiscrete-poisson}
	On $V_h\setminus\{0\}$, the bivector $L_2^h=Y_h\wedge J_h$, with $Y_h(\bfu)=\tfrac{\bfu}{\norm{\bfu}_h^2}$, is Poisson and satisfies $L_2^h\,\grad F_h=J_h$.
\end{theorem}

\begin{proof}
	By \eqref{eq:Jh-properties}, $\dd F_h(Y_h)=1$ and $\dd F_h(J_h)=0$, so \eqref{eq:bivector-action} gives $L_2^h\,\grad F_h=J_h$. Quadratic homogeneity of $J_h$ and $(\bfu,J_h)_h=0$ give
	\begin{equation*}
		DJ_h[Y_h]=\tfrac{2J_h}{\norm{\bfu}_h^2},
		\quad
		DY_h[J_h]
		=\tfrac{J_h}{\norm{\bfu}_h^2}-\tfrac{2(\bfu,J_h)_h\,\bfu}{\norm{\bfu}_h^4}
		=\tfrac{J_h}{\norm{\bfu}_h^2}.
	\end{equation*}
	Hence $[Y_h,J_h]=\tfrac{J_h}{\norm{\bfu}_h^2}$, so \Cref{prop:rank-two-poisson} applies.
\end{proof}

With $\bfu^{n+1/2}=\tfrac12(\bfu^n+\bfu^{n+1})$, the implicit midpoint discretization of \eqref{eq:NS-semidiscrete} is
\begin{equation}
	\label{eq:NS-midpoint}
	\tfrac{\bfu^{n+1}-\bfu^n}{\tau}
	=\nu P_h\Delta_h\bfu^{n+1/2}+J_h(\bfu^{n+1/2}) .
\end{equation}
\begin{theorem}
	\label{thm:ns-fully-discrete}
	At a nonzero midpoint, $J_h=L_2^h\,\grad F_h$ with $L_2^h$ Poisson; at a zero midpoint, set $L_2^h=0$, which also gives $J_h=0$. Hence \eqref{eq:NS-midpoint} is a fully discrete isothermal GENERIC discretization and satisfies the free-energy law
	\begin{equation}
		\label{eq:NS-discrete-energy}
		F_h(\bfu^{n+1})-F_h(\bfu^n)
		=-\nu\tau\norm{\grad_h\bfu^{n+1/2}}_h^2 .
	\end{equation}
\end{theorem}

\begin{proof}
	Pair \eqref{eq:NS-midpoint} with $\tau\,\bfu^{n+1/2}$; quadraticity of $F_h$ and \eqref{eq:Jh-properties} yield \eqref{eq:NS-discrete-energy}.
\end{proof}

\subsection{Cahn--Hilliard--Navier--Stokes equations}
\label{subsec:chns}

As a second isothermal free-energy example, consider on a periodic domain $\Om$ the CHNS system \cite{Kay-2007-CHNS,Gong-2018-CHNS,HanWang2015,ShenYang2010,ShenYang2015,Gong-2018-StaggeredCHNS,Chen-2020-CHNS,Zhao-2021-CHNS,Yang-2019-JCP,Yang-2025-JCP}
\begin{equation}
	\label{eq:chns-system}
	\left\{
	\begin{aligned}
		 & \partial_t\bfu+P_\sigma (\bfu\cdot\grad)\bfu+P_\sigma \phi\grad\bar\mu=\nu P_\sigma\Delta\bfu, \\
		 & \partial_t\phi+\diver(\phi\bfu)=m\Delta\bar\mu,                                                \\
		 & \mu=-\gamma\varepsilon\Delta\phi+\tfrac{\gamma}{\varepsilon}f'(\phi),                          \\
		 & \bar\mu=\mu-\tfrac{1}{\abs{\Om}}\int_\Om\mu\,\dd x,                                            \\
		 & f(\phi) = \tfrac{1}{4}(\phi^2 - 1)^2.
	\end{aligned}
	\right.
\end{equation}
with mobility $m>0$ and free energy
\begin{equation*}
	F(\bfu,\phi)=\tfrac12\norm{\bfu}^2
	+\tfrac{\gamma\varepsilon}{2}\norm{\grad\phi}^2
	+\tfrac{\gamma}{\varepsilon}(f(\phi),1).
\end{equation*}
On the divergence-free, fixed-mass affine phase space, the reversible field
\begin{equation*}
	J(\Phi)=-\bigl(P_\sigma(\bfu\cdot\grad)\bfu+P_\sigma\phi\grad\bar\mu,\ \diver(\phi\bfu)\bigr)
\end{equation*}
is free-energy neutral because $\grad F=(\bfu,\bar\mu)$ on this affine phase space and periodic integration by parts gives $\ip{\grad F}{J}=0$.

Set $D_2a^{n+1}:=\tfrac{3a^{n+1}-4a^n+a^{n-1}}{2\tau}$ and $\hat a^{n+1}:=2a^n-a^{n-1}$ for $a=\bfu,\phi$. In continuous spatial notation, with a compatible periodic summation-by-parts discretization understood, the GSGE--BDF2 scheme is
\begin{equation}
	\label{eq:chns-scheme}
	\left\{
	\begin{aligned}
		 & D_2\bfu^{n+1}+\grad p^{n+1}=\nu\Delta\bfu^{n+1}-\mathcal R_{\bfu}^{n+1},                                                 \\
		 & \diver\bfu^{n+1}=0,                                                                                                      \\
		 & D_2\phi^{n+1}=m\Delta\bar\mu^{n+1}-\mathcal R_{\phi}^{n+1},                                                           \\
		 & \mu^{n+1}=-\gamma\varepsilon\Delta\phi^{n+1}
		+\tfrac{\gamma}{\varepsilon} [\chi(\tfrac{3 \phi^{n+1} - \phi^n}{2}, \tfrac{3 \phi^n - \phi^{n-1}}{2}) - \hat{\phi}^{n+1}], \\
		 & \chi (a, b) := \tfrac{1}{4}(a^2 + b^2)(a + b),
	\end{aligned}
	\right.
\end{equation}
The ZEC residual uses the regularized gauge \eqref{eq:regularized-solution}, with $\sigma=\ell_3\tau^2$ and the differential weight specified below:
\begin{equation*}
	\bigl(\mathcal R_{\bfu}^{n+1},\mathcal R_{\phi}^{n+1}\bigr)
	=\lambda_A^{n+1}\,\hat{J}^{n+1}
	-\lambda_J^{n+1}\,A_{\ell_1,\ell_2} \hat q^{n+1},
\end{equation*}
where
\begin{equation*}
	\begin{aligned}
		\hat{J}^{n+1}
		               & =\bigl((\hat{\bfu}^{n+1}\!\cdot\!\grad)\hat{\bfu}^{n+1}+\hat{\phi}^{n+1}\grad\hat\mu^{n+1},\ \diver(\hat\phi^{n+1}\hat\bfu^{n+1})\bigr), \\
		\hat q^{n+1}
		               & =(\hat{\bfu}^{n+1},\bar{\hat\mu}^{n+1}),                                                                                                         \\
		A_{\ell_1,\ell_2} \hat q^{n+1}
		               & =\bigl(\ell_1^2\hat{\bfu}^{n+1}-\ell_2^2\Delta\hat{\bfu}^{n+1},\ \ell_1^2\bar{\hat\mu}^{n+1}-\ell_2^2\Delta\hat\mu^{n+1}\bigr),                                \\
		\hat \mu^{n+1} & = -\gamma \varepsilon \Delta \hat\phi^{n+1} + \tfrac{\gamma}{\varepsilon} \left((\hat{\phi}^{n+1})^3 - \hat{\phi}^{n+1}\right).
	\end{aligned}
\end{equation*}
Let $e^{n+1}=(\bfu^{n+1},\bar\mu^{n+1})$ and
\begin{equation}
	\label{eq:chns-denominator}
	\begin{aligned}
		d^{\,n+1}
		 & =\ip{A_{\ell_1,\ell_2}\hat q^{n+1}}{\hat q^{n+1}}+\ell_3\tau^2 \\
		 & =\ell_1^2\norm{\hat\bfu^{n+1}}^2+\ell_1^2\norm{\bar{\hat\mu}^{n+1}}^2
		+\ell_2^2\norm{\grad \hat\bfu^{n+1}}^2 \\
		 & \quad+\ell_2^2\norm{\grad\hat\mu^{n+1}}^2+\ell_3\tau^2.
	\end{aligned}
\end{equation}
Assume $d^{\,n+1}>0$ and set
\begin{equation}
	\label{eq:chns-lambdas}
	\lambda_A^{n+1}=\tfrac{\ip{A_{\ell_1,\ell_2}\hat q^{n+1}}{e^{n+1}}}{d^{\,n+1}},
	\quad
	\lambda_J^{n+1}=\tfrac{\ip{\hat J^{n+1}}{e^{n+1}}}{d^{\,n+1}} .
\end{equation}
\begin{remark}
	Let $\ell_1,\ell_2,\ell_3\ge0$ with $\ell_1^2+\ell_2^2>0$. If $\ell_3>0$, then $d^{\,n+1}>0$ even when the extrapolated force vanishes. Where the unregularized denominator is bounded away from zero, $\ell_3\tau^2=O(\tau^2)$ retains second-order consistency by \Cref{rem:regularized-gauge}. If $\ell_2=\ell_3=0$, the factors $\ell_1^2$ cancel and the closure reduces to the SGE--SBDF2 method of \cite{GuWangSGE2025}. If $\ell_1=\ell_3=0$, the phase component of $A_{\ell_1,\ell_2}\hat q^{n+1}$ is $-\ell_2^2\Delta\hat\mu^{n+1}$ and has zero mean without projecting $\hat\mu^{n+1}$; the displayed skew formula therefore gives a natural mass-preserving gauge whenever $d^{\,n+1}>0$, although the differential weight is semidefinite in this limit. The rank-two GSGE closure weakly couples the Navier--Stokes and phase-field subproblems only through the two scalar coefficients $\lambda_A^{n+1}$ and $\lambda_J^{n+1}$. Hence \eqref{eq:chns-scheme} admits the decoupled implementation of \cite{GuWangSGE2025}. Since the efficiency of the SGE framework and the BDF2 discretization has already been demonstrated in \cite{GuWangSGE2025,GuWangCSBDF22026}, we do not present additional numerical results here.
\end{remark}

\begin{proposition}
	\label{prop:chns-structure}
	The scheme \eqref{eq:chns-scheme}--\eqref{eq:chns-lambdas} satisfies
	\begin{equation*}
		(\mathcal R_\phi^{n+1},1)=0,
		\quad
		(\mathcal R_{\bfu}^{n+1},\bfu^{n+1})+(\mathcal R_{\phi}^{n+1},\bar\mu^{n+1})=0,
	\end{equation*}
	and $(\phi^{n+1},1)=(\phi^0,1)$ for all $n$ if $(\phi^1,1)=(\phi^0,1)$.
\end{proposition}

\begin{proof}
	Periodicity gives $(\diver(\hat\phi^{n+1}\hat\bfu^{n+1}),1)=0$ and $(\ell_1^2\bar{\hat\mu}^{n+1}-\ell_2^2\Delta\hat\mu^{n+1},1)=0$, hence $(\mathcal R_\phi^{n+1},1)=0$. Testing the phase equation with $1$ gives mass conservation. Moreover,
	\begin{equation*}
		\begin{aligned}
			(\mathcal R_{\bfu}^{n+1},\bfu^{n+1})
			+(\mathcal R_{\phi}^{n+1},\bar\mu^{n+1}) & =\ip{\lambda_A^{n+1}\hat J^{n+1}
					                                            -\lambda_J^{n+1}A_{\ell_1,\ell_2} \hat q^{n+1}}{e^{n+1}} \\
					                                         & =\lambda_A^{n+1}d^{\,n+1}\lambda_J^{n+1}
			-\lambda_J^{n+1}d^{\,n+1}\lambda_A^{n+1}=0 .
		\end{aligned}
	\end{equation*}
\end{proof}
\begin{theorem}
	\label{thm:chns-stability}
	Under the compatible spatial discretization and the condition $d^{\,n+1}>0$, the isothermal GSGE--BDF2 scheme satisfies
	\begin{equation*}
		\begin{aligned}
			 & \widetilde F^{\,n+1} - \widetilde F^{\,n} + \tau m \norm{\grad \mu^{n+1}}^2 + \nu \tau \norm{\grad \bfu^{n+1}}^2                      \\
			 & \quad + \tfrac{1}{4} \norm{\bfu^{n+1} - 2 \bfu^n + \bfu^{n-1}}^2 + \tfrac{\gamma \varepsilon}{4} \norm{\grad (\phi^{n+1} - 2\phi^n + \phi^{n-1})}^2 \\
			 & \quad + \tfrac{3 \gamma}{4 \varepsilon} \norm{\phi^{n+1} - 2\phi^n + \phi^{n-1}}^2 = 0.
		\end{aligned}
	\end{equation*}
	Here
	\begin{equation*}
		\begin{aligned}
			\widetilde F^{\,n+1}
			 & =\tfrac14\norm{\bfu^{n+1}}^2+\tfrac14\norm{2\bfu^{n+1}-\bfu^n}^2
			+\tfrac{\gamma\varepsilon}{4}\norm{\grad\phi^{n+1}}^2
			+\tfrac{\gamma\varepsilon}{4}\norm{\grad(2\phi^{n+1}-\phi^n)}^2     \\
			 & \quad+\tfrac{\gamma}{\varepsilon}\bigl(f(\tfrac{3\phi^{n+1} - \phi^n}{2}),1\bigr)
			+ \tfrac{3 \gamma}{8 \varepsilon} \norm{\phi^{n+1} - \phi^n}^2.
		\end{aligned}
	\end{equation*}
\end{theorem}

\begin{proof}
	Test the momentum, phase, and potential equations with $2\tau\bfu^{n+1}$, $2\tau\bar\mu^{n+1}$, and $2\tau D_2\phi^{n+1}$, respectively. The pressure term vanishes, and the time differences satisfy
	\begin{equation*}
		\begin{aligned}
			2\tau(D_2a^{n+1},a^{n+1})
			 & =\tfrac12\bigl[\norm{a^{n+1}}^2+
					            \norm{2a^{n+1}-a^n}^2\bigr]        \\
			 & \quad-\tfrac12\bigl[\norm{a^{n}}^2+
					                 \norm{2a^{n}-a^{n-1}}^2\bigr]
			+\tfrac12\norm{a^{n+1}-2a^n+a^{n-1}}^2.
		\end{aligned}
	\end{equation*}
	Mass conservation gives $(D_2\phi^{n+1},\bar\mu^{n+1})=(D_2\phi^{n+1},\mu^{n+1})$, while \Cref{prop:chns-structure} cancels the reversible terms:
	\begin{equation*}
		2\tau\bigl[(\mathcal R_{\bfu}^{n+1},\bfu^{n+1})
			+(\mathcal R_{\phi}^{n+1},\bar\mu^{n+1})\bigr]=0.
	\end{equation*}
	The remaining terms give the stated identity as in \cite{GuWangCSBDF22026}.
\end{proof}

\section{Conclusion}
\label{sec:conclusion}

GSGE represents the reversible action $L\,\grad E$ by contracting a two-form with $\grad S$, thereby writing full GENERIC as an entropy-driven generalized gradient flow. The admissible gauges form an affine family and retain the entropy-production law. Weighted least squares selects a unique representative, regularization removes the singular denominator at vanishing force profiles, and a projection-based gauge additionally preserves total energy and other prescribed invariants. For the isothermal reduction, the force is $\grad F$, and the rank-two Jacobi criterion determines when the reversible operator is Poisson.

For the incompressible Navier--Stokes equations, the MAC discretization satisfies the discrete rank-two Jacobi criterion, and the implicit midpoint rule yields a fully discrete isothermal GENERIC scheme with an exact free-energy law. For CHNS, the regularized GSGE--BDF2 scheme preserves mass, dissipates the discrete free energy unconditionally, and decouples through two scalar coefficients. Its parameter limits recover the SGE gauge and a gradient-weighted mass-preserving gauge.

\section*{Data availability}
No data were generated in this work.

\section*{Declaration of competing interest}
The authors declare that they have no known competing financial interests or personal relationships that could have appeared to influence the work reported in this paper.

\end{document}